\documentclass[12pt]{article}
\usepackage{amsmath,amscd,amsthm,a4,amssymb}

\def\dual{\,^{^{\complement}}\!}
\def\Rn{{\mathbb{R}^n}}

\def\i{\infty}
\def\a {\alpha}

 \newtheorem{thm}{Theorem}[section]
 \newtheorem{cor}[thm]{Corollary}
 \newtheorem{lem}[thm]{Lemma}
 
 \theoremstyle{definition}
 \newtheorem{defn}[thm]{Definition}
 \theoremstyle{remark}
 \newtheorem{rem}[thm]{Remark}
 
 \numberwithin{equation}{section}

\usepackage{amssymb}
\usepackage{color}

\def\Rn{{\mathbb{R}^n}}
\def\a {\alpha}
\def\i{\infty}

\def\L1loc{L_{\Phi}^{\rm loc}(\Rn)}

\def\dual{\,^{^{\complement}}\!}

\newcommand{\ess}{\mathop{\rm ess \; sup}\limits}
\newcommand{\es}{\mathop{\rm ess \; inf}\limits}
\begin{document}

\begin{center}
\LARGE On the Riesz potential and its commutators on generalized Orlicz-Morrey spaces
\end{center}

\

\centerline{\large Vagif S. Guliyev$^{a,b,}$\footnote{
E-mail adresses: vagif@guliyev.com (V.S. Guliyev), fderingoz@ahievran.edu.tr (F. Deringoz).}, Fatih Deringoz$^{a,1}$}

\

\centerline{$^{a}$\it Department of Mathematics, Ahi Evran University, Kirsehir, Turkey}

\centerline{$^{b}$\it Institute of Mathematics and Mechanics, Baku, Azerbaijan}

\

\begin{abstract}
We consider generalized Orlicz-Morrey spaces $M_{\Phi,\varphi}(\Rn)$ including their weak versions $WM_{\Phi,\varphi}(\Rn)$. In these spaces we prove the boundedness of the Riesz potential from $M_{\Phi,\varphi_1}(\Rn)$ to $M_{\Psi,\varphi_2}(\Rn)$ and from $M_{\Phi,\varphi_1}(\Rn)$ to $WM_{\Psi,\varphi_2}(\Rn)$. As applications of those results, the boundedness of the commutators of the Riesz potential on generalized Orlicz-Morrey space is also obtained. In all the cases the conditions  for the boundedness are given either in terms of Zygmund-type integral inequalities on $(\varphi_{1},\varphi_{2})$, which do not assume any assumption on monotonicity of $\varphi_{1}(x,r)$, $\varphi_{2}(x,r)$ in $r$.
\end{abstract}

\

\



\

\section{Introduction}

The theory of boundedness of classical operators of the real analysis, such as the maximal operator, fractional maximal operator, Riesz potential and the singular integral operators etc, have been
extensively investigated in various function spaces. Results on weak and strong type inequalities for operators of this kind in Lebesgue spaces are classical and can be found for example in \cite{BenSharp, St, Torch}. These boundedness extended to several function spaces which are generalizations of $L_{p}$-spaces, for example, Orlicz spaces, Morrey spaces, Lorentz spaces, Herz spaces, etc.

Orlicz spaces, introduced in \cite{Orlicz1, Orlicz2}, are generalizations of Lebesgue spaces $L_p$. They are useful tools in harmonic analysis and its applications. For example, the Hardy-Littlewood maximal operator is bounded on $L_p$ for $1 < p < \infty$, but not on $L_1$. Using Orlicz spaces, we can investigate the boundedness of
the maximal operator near $p = 1$ more precisely (see  \cite{Cianchi1, Kita1, Kita2}).

It is well known that the Riesz potential $I_{\alpha}$ of order $\a$ ($0 < \a <n$) plays an
important role in harmonic analysis, PDE and potential theory (see \cite{St}). Recall that $I_{\alpha}$ is defined by
$$
I_{\alpha} f(x)=\int_{\Rn}\frac{f(y)}{|x-y|^{n-\alpha} }dy,\qquad x\in\mathbb{R}^n.
$$
%where $C_{n,\a}=\frac{\Gamma((n-\a)/2)}{2^{\a}\pi^{n/2}\Gamma(\a/2)}$.

The classical result by Hardy-Littlewood-Sobolev states that if $1<p<q<\i$, then the operator $I_{\a}$ is bounded from $L_{p}(\Rn)$ to $L_{q}(\Rn)$ if and only if $\a=n\left(\frac{1}{p}-\frac{1}{q}\right)$ and for $p=1<q<\i$, the operator $I_{\a}$ is bounded from $L_{1}(\Rn)$ to $WL_{q}(\Rn)$ if and only if $\a=n\left(1-\frac{1}{q}\right)$. For boundedness of $I_{\alpha}$ on Morrey spaces $M_{p,\lambda}(\mathbb{R}^{n})$, see Peetre (Spanne)\cite{Peetre}, Adams \cite{Adams}.

The boundedness of $I_{\a}$ from Orlicz space $L_{\Phi}(\Rn)$ to $L_{\Psi}(\Rn)$ was studied by O'Neil \cite{O'Neil} and Torchinsky \cite{Torch1} under some restrictions involving the growths and certain monotonicity properties of $\Phi$ and $\Psi$. Moreover Cianchi \cite{Cianchi1} gave a necessary and sufficient condition for the boundedness of $I_{\a}$ from $L_{\Phi}(\Rn)$ to $L_{\Psi}(\Rn)$ and from $L_{\Phi}(\Rn)$ to weak Orlicz space $WL_{\Psi}(\Rn)$, which contain results above.

In \cite{DGS} the authors were study the boundedness of the maximal operator $M$ and the Calder\'{o}n-Zygmund operator $T$ from one generalized Orlicz-Morrey space $M_{\Phi,\varphi_1}(\Rn)$ to $M_{\Phi,\varphi_2}(\Rn)$ and from $M_{\Phi,\varphi_1}(\Rn)$ to the weak space $WM_{\Phi,\varphi_2}(\Rn)$.

Our definition of Orlicz-Morrey spaces (see Section \ref{secorlmor}) is different from that of Sawano et al. \cite{SawSugTan} and Nakai \cite{Nakai3, Nakai1}.

The main purpose of this paper is to find sufficient conditions on general Young functions $\Phi, \Psi$ and functions $\varphi_1$, $\varphi_2$ which ensure the boundedness of the Riesz potential $I_{\alpha}$ from one generalized Orlicz-Morrey spaces $M_{\Phi,\varphi_1}(\Rn)$ to another $M_{\Psi,\varphi_2}(\Rn)$, from $M_{\Phi,\varphi_1}(\Rn)$ to weak generalized Orlicz-Morrey spaces $WM_{\Psi,\varphi_2}(\Rn)$ and the boundedness of the commutator of the Riesz potential $[b,I_{\alpha}]$ from $M_{\Phi,\varphi_1}(\Rn)$ to $M_{\Psi,\varphi_2}(\Rn)$.

In the next section we recall the definitions of Orlicz and Morrey spaces and give the definition of Orlicz-Morrey and generalized Orlicz-Morrey spaces in Section 3. In Section 4 and Section 5 the results on the boundedness of the Riesz potential and its commutator on generalized Orlicz-Morrey spaces is obtained.

By $A\lesssim B$ we mean that $A\le CB$ with some positive constant $C$ independent
of appropriate quantities. If $A\lesssim B$ and $B\lesssim A$, we write $A\thickapprox B$ and say that $A$
and $B$ are equivalent.

\

\section{Some preliminaries on Orlicz and Morrey spaces}

In the study of local properties of solutions to of partial differential equations, together with weighted Lebesgue spaces,  Morrey spaces $M_{p,\lambda}(\Rn)$ play an important role, see \cite{Gi}. Introduced by C. Morrey  \cite{Morrey} in 1938, they are defined by the norm
\begin{equation*}
\left\| f\right\|_{M_{p,\lambda}}: = \sup_{x, \; r>0 } r^{-\frac{\lambda}{p}} \|f\|_{L_{p}(B(x,r))},
\end{equation*}
where $0 \le \lambda \le  n,$ $1\le p < \infty.$ Here and everywhere in the sequel $B(x,r)$ stands for the
ball in $\mathbb{R}^n$ of radius $r$ centered at $x$. Let $|B(x,r)|$ be the Lebesgue measure of the ball $B(x,r)$ and $|B(x,r)|=v_n r^n$, where $v_n$ is the volume of the unit ball in $\Rn$.

Note that $M_{p,0}=L_{p}(\Rn)$ and $M_{p,n}=L_{\infty}(\Rn)$. If $\lambda<0$ or $\lambda>n$, then $M_{p,\lambda }={\Theta}$,
where $\Theta$ is the set of all functions equivalent to $0$ on $\Rn$.

We also denote  by $WM_{p,\lambda} \equiv WM_{p,\lambda}(\Rn)$ the weak Morrey space of all functions $f\in WL_{p}^{\rm loc}(\Rn) $ for which
$$
\left\| f\right\|_{WM_{p,\lambda }} = \sup_{x\in \Rn, \; r>0} r^{-\frac{\lambda}{p}} \|f\|_{WL_{p}(B(x,r))} <\infty,
$$
where $WL_{p}(B(x,r))$ denotes the weak $L_p$-space.

We refer in particular to \cite{KJF} for the classical Morrey spaces.

We recall the definition of Young functions.

\begin{defn}\label{def2} A function $\Phi : [0,+\infty) \rightarrow [0,\infty]$ is called a Young function if $\Phi$ is convex, left-continuous, $\lim\limits_{r\rightarrow +0} \Phi(r) = \Phi(0) = 0$ and $\lim\limits_{r\rightarrow +\infty} \Phi(r) = \infty$.
\end{defn}

From the convexity and $\Phi(0) = 0$ it follows that any Young function is increasing.
If there exists $s \in (0,+\infty)$ such that $\Phi(s) = +\infty$, then $\Phi(r) = +\infty$ for $r \geq s$.

Let $\mathcal{Y}$ be the set of all Young functions $\Phi$ such that
\begin{equation*}
0<\Phi(r)<+\infty\qquad \text{for} \qquad 0<r<+\infty
\end{equation*}
If $\Phi \in \mathcal{Y}$, then $\Phi$ is absolutely continuous on every closed interval in $[0,+\infty)$
and bijective from $[0,+\infty)$ to itself.

\begin{defn} (Orlicz Space). For a Young function $\Phi$, the set
$$L_{\Phi}(\Rn)=\left\{f\in L_1^{loc}(\Rn): \int_{\Rn}\Phi(k|f(x)|)dx<+\infty
 \text{ for some $k>0$  }\right\}$$
is called Orlicz space. If $\Phi(r)=r^{p},\, 1\le p<\i$, then $L_{\Phi}(\Rn)=L_{p}(\Rn)$. If $\Phi(r)=0,\,(0\le r\le 1)$ and $\Phi(r)=\i,\,(r> 1)$, then $L_{\Phi}(\Rn)=L_{\i}(\Rn)$. The  space $L_{\Phi}^{\rm loc}(\Rn)$ endowed with the natural topology  is defined as the set of all functions $f$ such that  $f\chi_{_B}\in L_{\Phi}(\Rn)$ for all balls $B \subset \Rn$. We refer to the books \cite{KokKrbec, KrasnRut, RaoRen} for the theory of Orlicz Spaces.
\end{defn}

$L_{\Phi}(\Rn)$ is a Banach space with respect to the norm
$$\|f\|_{L_{\Phi}}=\inf\left\{\lambda>0:\int_{\Rn}\Phi\Big(\frac{|f(x)|}{\lambda}\Big)dx\leq 1\right\}.$$
We note that
$$\int_{\Rn}\Phi\Big(\frac{|f(x)|}{\|f\|_{L_{\Phi}}}\Big)dx\leq 1.$$
For a measurable set $\Omega\subset \mathbb{R}^{n}$, a measurable function $f$ and $t>0$, let
$$
m(\Omega,\ f,\ t)=|\{x\in\Omega:|f(x)|>t\}|.
$$
In the case $\Omega=\mathbb{R}^{n}$, we shortly denote it by $m(f,\ t)$.
\begin{defn} The weak Orlicz space
$$
WL_{\Phi}(\mathbb{R}^{n}):=\{f\in L_{1\mathrm{o}\mathrm{c}}^{1}(\mathbb{R}^{n}):\Vert f\Vert_{WL_{\Phi}}<+\infty\}
$$
is defined by the norm
$$
\Vert f\Vert_{WL_{\Phi}}=\inf\Big\{\lambda>0\ :\ \sup_{t>0}\Phi(t)m\Big(\frac{f}{\lambda},\ t\Big)\ \leq 1\Big\}.
$$
\end{defn}

For a Young function $\Phi$ and  $0 \leq s \leq +\infty$, let
$$\Phi^{-1}(s)=\inf\{r\geq 0: \Phi(r)>s\}\qquad (\inf\emptyset=+\infty).$$
If $\Phi \in \mathcal{Y}$, then $\Phi^{-1}$ is the usual inverse function of $\Phi$. We note that
\begin{equation}\label{younginverse}
\Phi(\Phi^{-1}(r))\leq r \leq \Phi^{-1}(\Phi(r)) \quad \text{ for } 0\leq r<+\infty.
\end{equation}
A Young function $\Phi$ is said to satisfy the $\Delta_2$-condition, denoted by  $\Phi \in \Delta_2$, if
$$
\Phi(2r)\le k\Phi(r) \text{    for } r>0
$$
for some $k>1$. If $\Phi \in \Delta_2$, then $\Phi \in \mathcal{Y}$. A Young function $\Phi$ is said to satisfy the $\nabla_2$-condition, denoted also by  $\Phi \in \nabla_2$, if
$$\Phi(r)\leq \frac{1}{2k}\Phi(kr),\qquad r\geq 0,$$
for some $k>1$. The function $\Phi(r) = r$ satisfies the $\Delta_2$-condition but does not satisfy the $\nabla_2$-condition.
If $1 < p < \infty$, then $\Phi(r) = r^p$ satisfies both the conditions. The function $\Phi(r) = e^r - r - 1$ satisfies the
$\nabla_2$-condition but does not satisfy the $\Delta_2$-condition.

For a Young function $\Phi$, the complementary function $\widetilde{\Phi}(r)$ is defined by
\begin{equation*}
\widetilde{\Phi}(r)=\left\{
\begin{array}{ccc}
\sup\{rs-\Phi(s): s\in [0,\infty)\}
& , & r\in [0,\infty), \\
+\infty&,& r=+\infty.
\end{array}
\right.
\end{equation*}
The complementary function  $\widetilde{\Phi}$ is also a Young function and $\widetilde{\widetilde{\Phi}}=\Phi$. If $\Phi(r)=r$, then $\widetilde{\Phi}(r)=0$ for $0\leq r \leq 1$ and $\widetilde{\Phi}(r)=+\infty$
  for $r>1$. If $1 < p < \infty$, $1/p+1/p^\prime= 1$ and $\Phi(r) =
r^p/p$, then $\widetilde{\Phi}(r) = r^{p^\prime}/p^\prime$. If $\Phi(r) = e^r-r-1$, then $\widetilde{\Phi}(r) = (1+r) \log(1+r)-r$. Note that $\Phi \in \nabla_2$ if and only if $\widetilde{\Phi} \in \Delta_2$. It is known that
\begin{equation}\label{2.3}
r\leq \Phi^{-1}(r)\widetilde{\Phi}^{-1}(r)\leq 2r \qquad \text{for } r\geq 0.
\end{equation}

Note that Young functions satisfy the properties
$$
\left\{
\begin{array}{ccc}
\Phi(\alpha t)\leq \alpha \Phi(t),
& \text{  if } &0\le\a\le1 \\
\Phi(\a t)\geq \a \Phi(t),&\text{  if }& \a>1
\end{array}
\right.
\text{ and }
\left\{
\begin{array}{ccc}
\Phi^{-1}(\alpha t)\geq \alpha \Phi^{-1}(t),
& \text{  if } &0\le\a\le1 \\
\Phi^{-1}(\a t)\leq \a \Phi^{-1}(t),&\text{  if }& \a>1.
\end{array}
\right.
$$
%\begin{equation*}
%\Phi(\alpha t)\leq \alpha \Phi(t)
%\end{equation*}
%for  all $0\le\a\le1$ and  $0 \le t < \i$, and
%\begin{equation*}
%\Phi(\beta t)\geq \beta \Phi(t)
%\end{equation*}
%for  all $\beta>1$ and  $0 \le t < \i$.

The following analogue of the H\"older inequality is known, see  \cite{Weiss}.
\begin{thm} \cite{Weiss} \label{HolderOr}
For a Young function $\Phi$ and its complementary function  $\widetilde{\Phi}$,
the following inequality is valid
$$\|fg\|_{L_{1}(\Rn)} \leq 2 \|f\|_{L_{\Phi}} \|g\|_{L_{\widetilde{\Phi}}}.$$
\end{thm}

The following lemma is valid.
\begin{lem}\label{lem4.0}  \cite{BenSharp, LiuWang}
Let $\Phi$ be a Young function and $B$ a set in $\mathbb{R}^n$ with finite Lebesgue measure. Then
$$
\|\chi_{_B}\|_{WL_{\Phi}(\Rn)} = \|\chi_{_B}\|_{L_{\Phi}(\Rn)} = \frac{1}{\Phi^{-1}\left(|B|^{-1}\right)}.
$$
\end{lem}

In the next sections where we prove our main estimates, we use the following lemma, which follows  from Theorem \ref{HolderOr}, Lemma \ref{lem4.0} and \eqref{2.3}.
\begin{lem}\label{lemHold}
For a Young function $\Phi$ and $B=B(x,r)$, the following inequality is valid
$$\|f\|_{L_{1}(B)} \leq 2 |B| \Phi^{-1}\left(|B|^{-1}\right) \|f\|_{L_{\Phi}(B)} .$$
\end{lem}

\

\section{Orlicz-Morrey and Generalized Orlicz-Morrey Spaces}\label{secorlmor}

\begin{defn} (Orlicz-Morrey Space). For a Young function $\Phi$ and $0 \le \lambda \le n$,
we denote by $M_{\Phi,\lambda}(\Rn)$ the Orlicz-Morrey space, the space of all
functions $f\in L_{\Phi}^{\rm loc}(\Rn)$ with finite quasinorm
$$
  \left\| f\right\|_{M_{\Phi,\lambda}}= \sup_{x\in \Rn, \; r>0 }
   \Phi^{-1}\big(r^{-\lambda}\big) \|f\|_{L_{\Phi}(B(x,r))}.
$$
\end{defn}
Note that $M_{\Phi,0}=L_{\Phi}(\Rn)$ and if $\Phi(r)=r^{p},\,1\le p<\i$, then $M_{\Phi,\lambda}(\Rn)=M_{p,\lambda}(\Rn)$.
%and $M_{\Phi,n}=L_{\infty}(\Rn)$. If $\lambda<0$ or $\lambda>n$, then $M_{\Phi,\lambda }={\Theta}$,
%where $\Theta$ is the set of all functions equivalent to $0$ on $\Rn$.

We also denote  by $WM_{\Phi,\lambda}(\Rn)$ the weak Orlicz-Morrey space of all functions $f\in WL_{\Phi}^{\rm loc}(\Rn) $ for which
$$
\left\| f\right\|_{WM_{\Phi,\lambda }} = \sup_{x\in \Rn, \; r>0 }
\Phi^{-1}\big(r^{-\lambda}\big) \|f\|_{WL_{\Phi}(B(x,r))} <\infty,
$$
where $WL_{\Phi}(B(x,r))$ denotes the weak $L_\Phi$-space of measurable functions
$f$ for which
\begin{align*}
\|f\|_{WL_{\Phi}(B(x,r))} & \equiv \|f \chi_{_{B(x,r)}}\|_{WL_{\Phi}(\Rn)}.
\end{align*}

\begin{defn} (generalized Orlicz-Morrey Space)
Let $\varphi(x,r)$ be a positive measurable function on $\Rn \times (0,\infty)$ and $\Phi$ any Young function.
We denote by $M_{\Phi,\varphi}(\Rn)$ the generalized Orlicz-Morrey space, the space of all
functions $f\in L_{\Phi}^{\rm loc}(\Rn)$ with finite quasinorm
$$
\|f\|_{M_{\Phi,\varphi}} = \sup\limits_{x\in\Rn, r>0}
\varphi(x,r)^{-1} \Phi^{-1}(|B(x,r)|^{-1}) \|f\|_{L_{\Phi}(B(x,r))}.
$$
Also by $WM_{\Phi,\varphi}(\Rn)$ we denote the weak generalized Orlicz-Morrey space of all functions $f\in WL_{\Phi}^{\rm loc}(\Rn)$ for which
$$
\|f\|_{WM_{p,\varphi}} = \sup\limits_{x\in\Rn, r>0} \varphi(x,r)^{-1} \Phi^{-1}(|B(x,r)|^{-1}) \|f\|_{WL_{\Phi}(B(x,r))} < \infty.
$$
\end{defn}

According to this definition, we recover the spaces $M_{\Phi,\lambda}$ and $WM_{\Phi,\lambda}$ under the choice $\varphi(x,r)=\frac{\Phi^{-1}\big(r^{-n}\big)}{\Phi^{-1}\big(r^{-\lambda}\big)}$:
$$
M_{\Phi,\lambda}=
M_{\Phi,\varphi}\Big|_{\varphi(x,r)=\frac{\Phi^{-1}\big(r^{-n}\big)}{\Phi^{-1}\big(r^{-\lambda}\big)}}, \quad
WM_{\Phi,\lambda}=
WM_{\Phi,\varphi}\Big|_{\frac{\Phi^{-1}\big(r^{-n}\big)}{\Phi^{-1}\big(r^{-\lambda}\big)}}.
$$

According to this definition, we recover the generalized Morrey spaces $M_{p,\varphi}$ and weak generalized Morrey spaces $WM_{p,\varphi}$ under the choice $\Phi(r)=r^{p},\,1\le p<\i$:
$$
M_{p,\varphi}=
M_{\Phi,\varphi}\Big|_{\Phi(r)=r^{p}}, \quad
WM_{p,\varphi}=
WM_{\Phi,\varphi}\Big|_{\Phi(r)=r^{p}}.
$$

\begin{rem}
There are different kinds of Orlicz-Morrey spaces in the literature. We want to make some comment about these spaces.

Let $\varphi:(0,\infty) \to (0,\infty)$
be a function
and
$\Phi:(0,\infty) \to (0,\infty)$
a Young function.
\begin{enumerate}
\item
For a cube $Q$, define $(\varphi,\Phi)$-average over $Q$ by
\[
\|f\|_{(\varphi,\Phi);Q}
:=
\inf\left\{
\lambda>0\,:\,
\frac{\varphi(|Q|)}{|Q|}\int_{Q}\Phi\left(\frac{|f(x)|}{\lambda}\right)\,dx
\le 1
\right\}
\]
and define its $\Phi$-average over $Q$ by
\[
\|f\|_{\Phi;Q}
:=
\inf\left\{
\lambda>0\,:\,
\frac{1}{|Q|}\int_{Q}\Phi\left(\frac{|f(x)|}{\lambda}\right)\,dx
\le 1 \right\}.
\]
\item
Define
\[
\|f\|_{{\mathcal L}_{\varphi,\Phi}}
:=
\sup_{Q \in {\mathcal Q}}
\|f\|_{(\varphi,\Phi);Q}.
\]
The function space ${\mathcal L}_{\varphi,\Phi}$
is defined to be the Orlicz-Morrey space
of the first kind as the set of all measurable functions
$f$ for which the norm $\|f\|_{{\mathcal L}_{\varphi,\Phi}}$
is finite.
\item
Define
\[
\|f\|_{{\mathcal M}_{\varphi,\Phi}}
:=
\sup_{Q \in {\mathcal Q}}
\varphi(|Q|)\|f\|_{\Phi;Q}.
\]
The function space ${\mathcal M}_{\varphi,\Phi}$
is defined to be the Orlicz-Morrey space
of the second kind
as the set of all measurable functions
$f$ for which the norm $\|f\|_{{\mathcal M}_{\varphi,\Phi}}$
is finite.
\end{enumerate}

According to our best knowledge,
it seems that ${\mathcal L}_{\varphi,\Phi}$
is more investigated than ${\mathcal M}_{\varphi,\Phi}$.
The space ${\mathcal L}_{\varphi,\Phi}$ is investigated
in \cite{LNYZ,LY,LT,MMOS13-1,MMOS13-2,MNOS10,MNOS12,Nakai3,N08-1,Nakai1,N11,SST06}
and the space ${\mathcal M}_{\varphi,\Phi}$ is investigated
in \cite{GST12,GST13,ISST12,SawSugTan}.

\end{rem}

\

\section{ Boundedness of the Riesz Potential in generalized Orlicz-Morrey spaces}

In this section sufficient conditions on the pairs $(\varphi_1, \varphi_2)$ and $(\Phi, \Psi)$ for the boundedness of $I_{\a}$ from one generalized Orlicz-Morrey spaces $M_{\Phi,\varphi_1}(\Rn)$ to another $M_{\Psi,\varphi_2}(\Rn)$ and from $M_{\Phi,\varphi_1}(\Rn)$ to the weak space $WM_{\Psi,\varphi_2}(\Rn)$  have been obtained.

Necessary and sufficient conditions on $(\Phi, \Psi)$ for the boundedness of $I_{\a}$ from $L_{\Phi}(\Rn)$ to $L_{\Psi}(\Rn)$ and $L_{\Phi}(\Rn)$ to $WL_{\Psi}(\Rn)$
have been obtained in \cite[Theorem 2]{Cianchi1}. In the statement of the theorem, $\Psi_{p}$ is the Young function associated with the Young function $\Psi$ and $p\in(1,\i]$ whose Young conjugate is given by
\begin{equation}\label{chi2.3}
\widetilde{\Psi_{p}}(s)=\int_{0}^{s}r^{p^{\prime}-1}(\mathcal{B}_{p}^{-1}(r^{p^{\prime}}))^{p^{\prime}}dr
\end{equation}
where
\begin{center}
$\mathcal{B}_{p}(s)= \int_{0}^{s}\frac{\Psi(t)}{t^{1+p^{\prime}}}dt$
\end{center}
and $p^{\prime}$, the Holder conjugate of $p$, equals either $p/(p-1)$ or 1, according to whether $ p<\infty$ or $ p=\infty$ and $\Phi_{p}$ denotes the Young function defined by
\begin{equation}\label{chi2.13}
\Phi_{p}(s)= \int_{0}^{s}r^{p^{\prime}-1}(\mathcal{A}_{p}^{-1}(r^{p'}))^{p'}dr
\end{equation}
where
\begin{center}
$\mathcal{A}_{p}(s)= \int_{0}^{s}\frac{\widetilde{\Phi}(t)}{t^{1+p^{\prime}}}dt$.
\end{center}

Recall that, if $\Phi$ and $\Psi$ are functions from $[0,\i)$ into $[0,\i]$, then $\Psi$ is said to dominate $\Phi$ globally if a positive constant $c$ exists such that $\Phi(s)\le\Psi(cs)$ for all $s\geq0$.

\begin{thm}\label{bounFrMaxOrl} \cite{Cianchi1}
Let $0<\a<n$. Let $\Phi$ and $\Psi$ Young functions and let $\Phi_{n/\alpha}$ and $\Psi_{n/\alpha}$ be the Young functions defined as in \eqref{chi2.13} and \eqref{chi2.3}, respectively. Then

(i)The Riesz potential $I_{\a}$ is bounded from $L_{\Phi}(\Rn)$ to $WL_{\Psi}(\Rn)$ if and only if
\begin{equation}\label{condweak}
\int_{0}^{1}\widetilde{\Phi}(t)/t^{1+n/(n-\alpha)}dt<\infty \text{  and $\Phi_{n/\alpha}$ dominates $\Psi$ globally.}
\end{equation}

(ii) The Riesz potential $I_{\a}$ is bounded from $L_{\Phi}(\Rn)$ to $L_{\Psi}(\Rn)$ if and only if
\begin{align}\label{condstr}
&\int_{0}^{1}\widetilde{\Phi}(t)/t^{1+n/(n-\alpha)}dt<\i, ~~ \int_{0}^{1}\Psi(t)/t^{1+n/(n-\alpha)}dt<\i, \notag
\\
& \text{ $\Phi$ dominates $\Psi_{n/\alpha}$  globally and $\Phi_{n/\alpha}$ dominates $\Psi$ globally.}
\end{align}
\end{thm}

We will  use the following statement on the boundedness of the weighted Hardy operator
$$
H_{w} g(t):=\int_t^{\infty} g(s) w(s) ds,~ \ \  0<t<\infty,
$$
where $w$ is a weight.

The following theorem was proved in \cite{GulJMS2013} (see, also \cite{DGS}).

\begin{thm}\label{thm3.2.}
Let $v_1$, $v_2$ and $w$ be weights on  $(0,\infty)$ and $v_1(t)$ be bounded outside a neighborhood of the
origin. The inequality
\begin{equation} \label{vav01}
\ess_{t>0} v_2(t) H_{w} g(t) \leq C \ess_{t>0} v_1(t) g(t)
\end{equation}
holds for some $C>0$ for all non-negative and non-decreasing $g$ on $(0,\i)$ if and
only if
\begin{equation*}
B:= \sup _{t>0} v_2(t)\int_t^{\infty} \frac{w(s) ds}{\ess_{s<\tau<\infty} v_1(\tau)}<\infty.
\end{equation*}
Moreover, the value  $C=B$ is the best constant for  \eqref{vav01}.
\end{thm}

\begin{lem}\label{lemconCian}
Let $\Phi$ and $\Psi$ Young functions and $\Phi_{p}$, $p\in(1,\i]$, Young function defined as in \eqref{chi2.13}. If $\int_{0}^{1}\widetilde{\Phi}(t)/t^{1+p^{\prime}}dt<\infty$ and $\Phi_{p}$ dominates $\Psi$ globally then
\begin{equation}\label{framaxint}
\Phi^{-1}(r)\lesssim r^{\frac{1}{p}}\Psi^{-1}(r),\qquad \text{ for } r>0
\end{equation}
\end{lem}

\begin{proof}
If $\int_{0}^{1}\widetilde{\Phi}(t)/t^{1+p^{\prime}}dt<\infty$ then
\begin{equation}\label{cianci2 3.26}
1\le 2r^{-\frac{1}{p^{\prime}}}\widetilde{\Phi}^{-1}(r)\Phi_{p}^{-1}(r),\qquad \text{ for } r>0.
\end{equation}
For the proof of this claim see \cite[p.~50]{Cianchi2}.

If $\Phi_{p}$ dominates $\Psi$ globally, then a positive constant $C$ exist such that
\begin{equation}\label{cianci1 2.19}
\Phi_{p}^{-1}(r)\le C\Psi^{-1}(r),\qquad \text{ for } r>0
\end{equation}
Indeed,
\begin{eqnarray*}
\Psi^{-1}(r)&=&\inf\{t\geq0: \Psi(t)>r\}\\
{}&\geq&\inf\{t\geq0: \Phi_{p}(Ct)>r\}\\
{}&=&\frac{1}{C}\inf\{Ct\geq0: \Phi_{p}(Ct)>r\}\\
{}&=&\frac{1}{C}\Phi_{p}^{-1}(r).
\end{eqnarray*}
Thus, \eqref{framaxint} follows from \eqref{cianci2 3.26}, \eqref{cianci1 2.19} and \eqref{2.3}.
\end{proof}

The following lemma is valid.
\begin{lem}\label{lemGultecnOrl}
Let $0<\a<n$, $\Phi$ and $\Psi$ Young functions, $f\in L_{\Phi}^{\rm loc}(\Rn)$ and $B=B(x_0,r)$. If $(\Phi, \Psi)$ satisfy the conditions \eqref{condstr}, then
\begin{equation}\label{strGultecnOrl}
\|I_\alpha f\|_{L_{\Psi}(B)} \lesssim \frac{1}{\Psi^{-1}\big(r^{-n}\big)}
 \int_{2r}^{\i} \|f\|_{L_{\Phi}(B(x_0,t))} \Psi^{-1}\big(t^{-n}\big) \frac{dt}{t}
\end{equation}
and if $(\Phi, \Psi)$ satisfy the conditions \eqref{condweak}, then
\begin{equation}\label{weakGultecnOrl}
\|I_\alpha f\|_{WL_{\Psi}(B)} \lesssim \frac{1}{\Psi^{-1}\big(r^{-n}\big)}
 \int_{2r}^{\i} \|f\|_{L_{\Phi}(B(x_0,t))} \Psi^{-1}\big(t^{-n}\big) \frac{dt}{t}
\end{equation}
\end{lem}
\begin{proof}
Suppose that the conditions \eqref{condstr} satisfied. For arbitrary $x_0 \in\Rn$, set $B=B(x_0,r)$ for
the ball centered at $x_0$ and of radius $r$, $2B=B(x_0,2r)$.
 %Write $f=f_1+f_2$ with $f_1=f\chi_{2B}$ and $f_2=f\chi_{\dual {(2B)}}$.
We represent  $f$ as
\begin{equation*}
f=f_1+f_2, \ \quad f_1(y)=f(y)\chi _{2B}(y),\quad
 f_2(y)=f(y)\chi_{\dual {(2B)}}(y), \ \quad r>0,
\end{equation*}
and have
$$
\|I_\alpha f\|_{L_{\Psi}(B)} \le \|I_\alpha f_1\|_{L_{\Psi}(B)} +\|I_\alpha f_2\|_{L_{\Psi}(B)}.
$$

Since $f_1\in L_{\Phi}(\Rn)$, $I_\alpha f_1\in L_{\Psi}(\Rn)$  and from the
boundedness of $I_\alpha$ from $L_{\Phi}(\Rn)$ to $L_{\Psi}(\Rn)$ (see Theorem \ref{bounFrMaxOrl}) it follows that:
\begin{equation*}
\|I_\alpha f_1\|_{L_\Psi(B)}\leq \|I_\alpha f_1\|_{L_\Psi(\Rn)}\leq
C\|f_1\|_{L_\Phi(\Rn)}=C\|f\|_{L_\Phi(2B)},
\end{equation*}
where constant $C >0$ is independent of $f$.

It's clear that $x\in B$, $y\in \dual{(2B)}$ implies
$\frac{1}{2}|x_0-y|\le |x-y|\le\frac{3}{2}|x_0-y|$. We get
$$
|I_\alpha f_2(x)|\leq 2^{n-\alpha} \int_{\dual {(2B)}}\frac{|f(y)|}{|x_0-y|^{n-\alpha}}dy.
$$
By Fubini's theorem we have
\begin{equation*}
\begin{split}
\int_{\dual {(2B)}}\frac{|f(y)|}{|x_0-y|^{n-\alpha}}dy &
\thickapprox
\int_{\dual {(2B)}}|f(y)|\int_{|x_0-y|}^{\i}\frac{dt}{t^{n+1-\alpha}}dy
\\
&\thickapprox \int_{2r}^{\i}\int_{2r\leq |x_0-y|< t}|f(y)|dy\frac{dt}{t^{n+1-\alpha}}
\\
&\lesssim \int_{2r}^{\i}\int_{B(x_0,t) }|f(y)|dy\frac{dt}{t^{n+1-\alpha}}.
\end{split}
\end{equation*}
By Lemma \ref{lemHold} and Lemma \ref{lemconCian} for $p=n/\a$ we get
\begin{eqnarray} \label{sal00}
\int_{\dual {(2B)}}\frac{|f(y)|}{|x_0-y|^{n}}dy & \lesssim & \int_{2r}^{\i}\|f\|_{L_{\Phi}(B(x_0,t))} \Phi^{-1}\big(t^{-n}\big) t^{\alpha-1} dt\nonumber\\
{}& \lesssim & \int_{2r}^{\i}\|f\|_{L_{\Phi}(B(x_0,t))} \Psi^{-1}\big(t^{-n}\big) \frac{dt}{t}
\end{eqnarray}

Moreover,
\begin{equation} \label{ves2K}
\|I_\alpha f_2\|_{L_{\Psi}(B)}\lesssim
\frac{1}{\Psi^{-1}\big(r^{-n}\big)}\int_{2r}^{\i}\|f\|_{L_{\Phi}(B(x_0,t))} \Psi^{-1}\big(t^{-n}\big) \frac{dt}{t}
\end{equation}
is valid. Thus
\begin{equation*}
\|I_\alpha f\|_{L_\Psi(B)}\lesssim \|f\|_{L_{\Phi}(2B)}+
\frac{1}{\Psi^{-1}\big(r^{-n}\big)}\int_{2r}^{\i}\|f\|_{L_\Phi (B(x_0,t))} \Psi^{-1}\big(t^{-n}\big) \frac{dt}{t}.
\end{equation*}

On the other hand, using the property of Young function as it mentioned in \eqref{2.3}
\begin{eqnarray*}
\Psi^{-1}\big(r^{-n}\big) & \thickapprox & \Psi^{-1}\big(r^{-n}\big) r^{n} \int_{2r}^{\i}\frac{dt}{t^{n+1}}
\\
& \lesssim & \int_{2r}^{\i} \Psi^{-1}\big(t^{-n}\big) \frac{dt}{t}
\end{eqnarray*}
and we get
\begin{equation} \label{ves2F}
\|f\|_{L_{\Phi}(2B)}\lesssim
\frac{1}{\Psi^{-1}\big(r^{-n}\big)}\int_{2r}^{\i} \|f\|_{L_{\Phi}(B(x_0,t))} \Psi^{-1}\big(t^{-n}\big) \frac{dt}{t}.
\end{equation}

Thus
\begin{equation*}
\|I_\alpha f\|_{L_{\Psi}(B)}\lesssim \frac{1}{\Psi^{-1}\big(r^{-n}\big)} \int_{2r}^{\i}\|f\|_{L_\Phi (B(x_0,t))} \Psi^{-1}\big(t^{-n}\big) \frac{dt}{t}.
\end{equation*}

Suppose that the conditions \eqref{condstr} satisfied. From the boundedness of $I_{\alpha}$ from $L_{\Phi}(\Rn)$ to $WL_{\Psi}(\Rn)$ (see Theorem \ref{bounFrMaxOrl}) and \eqref{ves2F} it follows
that:
\begin{equation} \label{ggg1}
\begin{split}
\|I_{\alpha}f_1\|_{WL_\Psi(B)} & \leq \|I_{\alpha}f_1\|_{WL_\Psi(\Rn)} \lesssim  \|f_1\|_{L_\Phi(\Rn)}
\\
& = \|f\|_{L_\Phi(2B)} \lesssim  \frac{1}{\Psi^{-1}\big(r^{-n}\big)}\int_{2r}^{\i} \|f\|_{L_{\Phi}(B(x_0,t))} \Psi^{-1}\big(t^{-n}\big) \frac{dt}{t}.
\end{split}
\end{equation}
Then by \eqref{ves2K} and \eqref{ggg1} we get the inequality \eqref{weakGultecnOrl}.
\end{proof}

\begin{thm}\label{thm4.4.}
Let $0<\a<n$ and the functions $(\varphi_1,\varphi_2)$ and $(\Phi, \Psi)$ satisfy the condition
\begin{equation}\label{eq3.6.VZfrMax}
\int_{r}^{\i} \es_{t<s<\i}\frac{\varphi_1(x,s)}{\Phi^{-1}\big(s^{-n}\big)}\Psi^{-1}\big(t^{-n}\big)\frac{dt}{t}  \le C \, \varphi_2(x,r),
\end{equation}
where $C$ does not depend on $x$ and $r$.
Then for the conditions \eqref{condstr}, $I_{\a}$ is bounded from $M_{\Phi,\varphi_1}(\Rn)$ to $M_{\Psi,\varphi_2}(\Rn)$
and for the conditions \eqref{condweak}, $I_{\a}$ is bounded from $M_{\Phi,\varphi_1}(\Rn)$ to $WM_{\Psi,\varphi_2}(\Rn)$.
\end{thm}
\begin{proof}
By Lemma \ref{lemGultecnOrl} and Theorem \ref{thm3.2.} we get
\begin{equation*}
\begin{split}
\|I_{\a} f\|_{M_{\Psi,\varphi_2}} & \lesssim \sup\limits_{x\in\Rn, r>0}
\varphi_2(x,r)^{-1}\,
\int_{r}^{\i} \Psi^{-1}\big(t^{-n}\big) \|f\|_{L_{\Phi}(B(x,t))}\frac{dt}{t}
\\
& \lesssim \sup\limits_{x\in\Rn, r>0}
\varphi_1(x,r)^{-1} \Phi^{-1}\big(r^{-n}\big) \|f\|_{L_{\Phi}(B(x,r))}
\\
& = \|f\|_{M_{\Phi,\varphi_1}},
\end{split}
\end{equation*}
if \eqref{condstr} satisfied and
\begin{equation*}
\begin{split}
\|I_{\a} f\|_{WM_{\Psi,\varphi_2}} & \lesssim \sup\limits_{x\in\Rn, r>0}
\varphi_2(x,r)^{-1}\,
\int_{r}^{\i} \Psi^{-1}\big(t^{-n}\big) \|f\|_{L_{\Phi}(B(x,t))}\frac{dt}{t}
\\
& \lesssim \sup\limits_{x\in\Rn, r>0}
\varphi_1(x,r)^{-1} \Phi^{-1}\big(r^{-n}\big) \|f\|_{L_{\Phi}(B(x,r))}
\\
& = \|f\|_{M_{\Phi,\varphi_1}},
\end{split}
\end{equation*}
if \eqref{condweak} satisfied.
\end{proof}
\begin{rem}
Recall that, for $0<\a<n$,
$$
M_{\a}f(x)\le v_{n}^{\frac{\a}{n}-1}I_{\a}(|f|)(x),
$$
hence Theorem \ref{thm4.4.} imply the boundedness of the fractional maximal operator $M_{\a}$ from $M_{\Phi,\varphi_1}(\Rn)$ to $M_{\Psi,\varphi_2}(\Rn)$ and from $M_{\Phi,\varphi_1}(\Rn)$ to $WM_{\Psi,\varphi_2}(\Rn)$.
\end{rem}
If we take $\Phi(t)=t^{p},\,\Psi(t)=t^{q},\,1\le p,q<\i$ at Theorem \ref{thm4.4.} we get following corollary which was proved in \cite{GULAKShIEOT2012} and containing results obtained in \cite{GulDoc, GulBook, GulJIA, Miz, Nakai}.
\begin{cor}\label{bounRszGenMor}
Let $0<\a<n$, $1\le p < \frac{n}{\a}$, $\frac{1}{q}=\frac{1}{p}-\frac{\a}{n}$ and $(\varphi_{1},\varphi_{2})$ satisfy the condition
\begin{equation*}%\label{GulJIA5.13}
\int_{r}^{\i}\frac{\es_{t<s<\i}\varphi_{1}(x,s)s^{\frac{n}{p}}}{t^{\frac{n}{q}+1}}dt\le C\varphi_{2}(x,r),
\end{equation*}
where $C$ does not depend on $x$ and $r$. Then $I_{\alpha}$ is bounded from $M_{p,\varphi_{1}}$ to $M_{q,\varphi_{2}}$ for $p>1$ and from $M_{1,\varphi_{1}}$ to $WM_{q,\varphi_{2}}$ for $p=1$.
\end{cor}

In the case $\varphi_1(x,r)=\frac{\Phi^{-1}\big(r^{-n}\big)}{\Phi^{-1}\big(r^{-\lambda_1}\big)}$, $\varphi_2(x,r)=\frac{\Psi^{-1}\big(r^{-n}\big)}{\Psi^{-1}\big(r^{-\lambda_2}\big)}$  from Theorem \ref{thm4.4.} we get the following Spanne type theorem for the boundedness of the Riesz potential on Orlicz-Morrey spaces.
\begin{cor} \label{ggffdd}
Let $0<\a<n$, $\Phi$ and $\Psi$ be Young functions, $0 \le \lambda_1, \lambda_2 < n$ and $(\Phi, \Psi)$ satisfy the condition
\begin{equation*}%\label{eq3.6.VZfrMaxcor}
\int_{r}^{\i} \frac{\Psi^{-1}\big(t^{-n}\big)}{\Phi^{-1}\big(t^{-\lambda_1}\big)}\frac{dt}{t}  \le C \, \frac{\Psi^{-1}\big(r^{-n}\big)}{\Psi^{-1}\big(r^{-\lambda_2}\big)},
\end{equation*}
where $C$ does not depend on $r$.
Then for the conditions \eqref{condstr}, $I_{\a}$ is bounded from $M_{\Phi,\lambda_1}(\Rn)$ to $M_{\Psi,\lambda_2}(\Rn)$
and for the conditions \eqref{condweak}, $I_{\a}$ is bounded from $M_{\Phi,\lambda_1}(\Rn)$ to $WM_{\Psi,\lambda_2}(\Rn)$.
\end{cor}

\begin{rem}
If we take $\Phi(t)=t^{p},\,\Psi(t)=t^{q},\,1\le p,q<\i$ at Corollary \ref{ggffdd} we get Spanne type boundedness of $I_{\a}$, i.e. if $0<\a<n$, $1<p<\frac{n}{\a}$, $0<\lambda<n-\a p$, $\frac{1}{p}-\frac{1}{q}=\frac{\a}{n}$ and $\frac{\lambda}{p}=\frac{\mu}{q}$, then for $p>1$ the Riesz potential $I_{\a}$ is bounded from $M_{p,\lambda}(\Rn)$ to $M_{q,\mu}(\Rn)$ and for $p=1$, $I_{\a}$ is bounded from $M_{1,\lambda}(\Rn)$ to $WM_{q,\mu}(\Rn)$.
\end{rem}

%\tcr{In the case $\varphi_1(x,r)=\varphi_2(x,r)=  \frac{1}{\Phi^{-1}\big(r^{-\lambda}\big)}$  of Orlicz-Morrey spaces from Corollary \ref{ggffdd} we get}
%\begin{cor}
%\tcr{Let a Young function  $\Phi \in \Delta_2$, $0 \le \lambda \le n$.
%Then the maximal operator $M$ is bounded from $M_{\Phi,\lambda}(\Rn)$ to $M_{\Phi,\lambda}(\Rn)$ and is bounded from $M_{\Phi,\lambda}(\Rn)$ to $WM_{\Phi,\lambda}(\Rn)$.}
%\end{cor}

\section{Commutators of Riesz potential in the spaces $M_{\Phi,\varphi}$}

For a function $b\in L_1^{\rm loc}(\Rn)$, let $M_{b}$ be the corresponding multiplication operator defined by $M_{b}f=bf$ for measurable function $f$. Let $T$ be the classical Calder\'{o}n-Zygmund singular integral operator, then the commutator between $T$ and $M_{b}$ is denoted by $[b,T]:=M_{b}T-TM_{b}$. A famous theorem of Coifman et al. \cite{CRW} gave
a characterization of $L_p$-boundedness of $[b,T]$ when T are the Riesz transforms $R_j$ $(j = 1, \dots ,n)$. Using this characterization, the authors of \cite{CRW} got a decomposition
theorem of the real Hardy spaces. The boundedness result was generalized to other contexts and important applications to some non-linear PDEs were given by Coifman
et al. \cite{CLMS}.

We recall the definition of the space of $BMO(\Rn)$.

\begin{defn}
Suppose that $f\in L_1^{\rm loc}(\Rn)$, let
\begin{equation*}
\|f\|_\ast=\sup_{x\in\Rn, r>0}\frac{1}{|B(x,r)|}
\int_{B(x,r)}|f(y)-f_{B(x,r)}|dy<\infty,
\end{equation*}
where
$$
f_{B(x,r)}=\frac{1}{|B(x,r)|} \int_{B(x,r)} f(y)dy.
$$
Define
$$
BMO(\Rn)=\{ f\in L_1^{\rm loc}(\Rn) ~ : ~ \| f \|_{\ast} < \infty  \}.
$$
\end{defn}
%If one regards two functions whose difference is a constant as one, then space

Modulo constants, the space $BMO(\Rn)$ is a Banach space with respect to the norm $\| \cdot \|_{\ast}$.

\begin{rem} \label{rem2.4.}
$(1)~~$ The John--Nirenberg inequality:
there are constants $C_1$, $C_2>0$, such that for all $f \in BMO(\Rn)$ and $\beta>0$
$$
\left| \left\{ x \in B \, : \, |f(x)-f_{B}|>\beta \right\}\right|
\le C_1 |B| e^{-C_2 \beta/\| f \|_{\ast}}, ~~~ \forall B \subset \Rn.
$$

$(2)~~$ The John--Nirenberg inequality implies that
\begin{equation} \label{lem2.4.}
\|f\|_\ast \thickapprox \sup_{x\in\Rn, r>0}\left(\frac{1}{|B(x,r)|}
\int_{B(x,r)}|f(y)-f_{B(x,r)}|^p dy\right)^{\frac{1}{p}}
\end{equation}
for $1<p<\infty$.

$(3)~~$ Let $f\in BMO(\Rn)$. Then there is a constant $C>0$ such that
\begin{equation} \label{propBMO}
\left|f_{B(x,r)}-f_{B(x,t)}\right| \le C \|f\|_\ast \ln \frac{t}{r} \;\;\; \mbox{for} \;\;\; 0<2r<t,
\end{equation}
where $C$ is independent of $f$, $x$, $r$ and $t$.
\end{rem}

\begin{defn}
A Young function $\Phi$ is said to be of upper type p (resp. lower type p) for some $p\in[0,\i)$, if there exists a positive constant $C$ such that, for all $t\in[1,\i)$(resp. $t\in[0,1]$) and $s\in[0,\i)$,
$$
\Phi(st)\le Ct^p\Phi(s).
$$
\end{defn}

\begin{rem}\label{remlowup}
We know that if $\Phi$ is lower type $p_0$ and upper type $p_1$ with $1<p_0\le p_1<\i$, then $\Phi\in \Delta_2\cap\nabla_2$. Conversely if $\Phi\in \Delta_2\cap\nabla_2$, then  $\Phi$ is lower type $p_0$ and upper type $p_1$ with $1<p_0\le p_1<\i$ (see \cite{KokKrbec}).
\end{rem}

\begin{lem}\cite{Ky1}\label{Kylowupp}
Let $\Phi$ be a Young function which is lower type $p_0$ and upper type $p_1$ with $1\le p_0\le p_1<\i$. Let $\widetilde{C}$ be a positive constant. Then there exists a positive constant $C$ such that for any ball $B$ of $\Rn$ and $\mu\in(0,\i)$
$$\int_{B}\Phi\left(\frac{|f(x)|}{\mu}\right)dx\le \widetilde{C}$$
implies that $\|f\|_{L_\Phi(B)}\le C\mu$.
\end{lem}
In the following lemma we provide a generalization of the property \eqref{lem2.4.} from $L_p$-norms to Orlicz norms.
\begin{lem}\label{Bmo-orlicz}
Let $f\in BMO(\Rn)$ and $\Phi$ be a Young function. Let $\Phi$ is lower type $p_0$ and upper type $p_1$ with $1\le p_0\le p_1<\i$, then
$$
\|f\|_\ast \thickapprox \sup_{x\in\Rn, r>0}\Phi^{-1}\big(r^{-n}\big)\left\|f(\cdot)-f_{B(x,r)}\right\|_{L_{\Phi}(B(x,r))}
$$
\end{lem}

\begin{proof}
By H\"{o}lder's inequality, we have
$$\|f\|_\ast \lesssim \sup_{x\in\Rn, r>0}\Phi^{-1}\big(r^{-n}\big)\left\|f(\cdot)-f_{B(x,r)}\right\|_{L_{\Phi}(B(x,r))}.$$

Now we show that
$$
\sup_{x\in\Rn, r>0}\Phi^{-1}\big(r^{-n}\big)\left\|f(\cdot)-f_{B(x,r)}\right\|_{L_{\Phi}(B(x,r))} \lesssim \|f\|_\ast .
$$
Without loss of generality, we may assume that $\|f\|_\ast=1$; otherwise,
we replace $f$ by $f/\|f\|_\ast$. By the fact that $\Phi$ is lower type $p_0$ and upper type $p_1$ and \eqref{younginverse} it follows that
$$
\int_{B(x,r)}\Phi\left(\frac{|f(y)-f_{B(x,r)}|\Phi^{-1}\big(|B(x,r)|^{-1}\big)}{\|f\|_\ast}\right)dy
$$
$$
=\int_{B(x,r)}\Phi\left(|f(y)-f_{B(x,r)}|\Phi^{-1}\big(|B(x,r)|^{-1}\big)\right)dy
$$
$$
\lesssim\frac{1}{|B(x,r)|}\int_{B(x,r)}\left[|f(y)-f_{B(x,r)}|^{p_0}+|f(y)-f_{B(x,r)}|^{p_1}\right]dy\lesssim 1.
$$
By Lemma \ref{Kylowupp} we get the desired result.
\end{proof}

\begin{rem}
Note that statements of  type of  Lemma \ref{Bmo-orlicz} are known
in a more general case of rearrangement invariant spaces and also
variable exponent Lebesgue spaces $L^{p(\cdot)}$, see  \cite{Kwok-Pun Ho} and \cite{IzukiSaw}, but we gave
a short proof of Lemma \ref{Bmo-orlicz} for completeness of presentation.
\end{rem}

\begin{defn}
Let $\Phi$ be a Young function. Let
$$
a_{\Phi}:=\inf_{t\in(0,\i)}\frac{t\Phi^{\prime}(t)}{\Phi(t)}, \qquad  b_{\Phi}:=\sup_{t\in(0,\i)}\frac{t\Phi^{\prime}(t)}{\Phi(t)}.
$$
\end{defn}

\begin{rem}\label{indorl}
It is known that $\Phi\in \Delta_2\cap\nabla_2$ if and only if $1<a_{\Phi}\le b_{\Phi}<\i$ (See \cite{KrasnRut}).
\end{rem}

\begin{rem}
Remark \ref{indorl} and Remark \ref{remlowup} show us that a Young function $\Phi$ is lower type $p_0$ and upper type $p_1$ with $1<p_0\le p_1<\i$ if and only if $1<a_{\Phi}\le b_{\Phi}<\i$.
\end{rem}

%Suppose that $b\in L_1^{\rm loc}(\Rn)$, then the commutator generated by $b$ and the
%Riesz potential $I_{\a}$ is defined by
%\begin{equation}
%[b,I_{\a}]f(x)=b(x)I_{\a}f(x)-I_{\a}(bf)(x)=\int_{\Rn}\frac{b(x)-b(y)}{|x-y|^{n-\a}}f(y)dy.
%\end{equation}

The characterization of $(L_p,L_q)$ boundedness of the commutator $[b,I_{\a}]$ between $M_b$ and $I_\a$ was given by Chanillo \cite{Chanillo}.

\begin{thm}\cite{Chanillo}
Let $0<\a<n$, $1< p < \frac{n}{\a}$ and $\frac{1}{q}=\frac{1}{p}-\frac{\a}{n}$. Then $[b,I_{\a}]$ is a bounded operator from $L_p(\Rn)$ to $L_q(\Rn)$ if and only if $b\in BMO(\Rn)$.
\end{thm}

The $(L_{\Phi},L_{\Psi})$ boundedness of the commutator $[b,I_{\a}]$ was given by Fu, Yang and Yuan \cite{FuYangYuan}.

\begin{thm}\cite{FuYangYuan}\label{comorlicz}
Let $0<\a<n$ and $b\in BMO(\Rn)$. Let $\Phi$ be a Young function and $\Psi$ defined, via its inverse, by setting, for all $t\in(0,\i)$, $\Psi^{-1}(t):=\Phi^{-1}(t)t^{-\a/n}$. If $1<a_{\Phi}\le b_{\Phi}<\i$ and $1<a_{\Psi}\le b_{\Psi}<\i$ then $[b,I_{\a}]$ is bounded from $L_{\Phi}(\Rn)$ to $L_{\Psi}(\Rn)$.
\end{thm}

We will  use the following statement on the boundedness of the weighted Hardy operator
$$
H^{\ast}_{w} g(r):=\int_r^{\infty} \left(1+\ln \frac{t}{r}\right) \, g(t) w(t) dt,~ \ \  r\in (0,\infty),
$$
where $w$ is a weight.

The following theorem was proved in \cite{GulEMJ2012}.

\begin{thm}\label{thm3.2.com}
Let $v_1$, $v_2$ and $w$ be weights on  $(0,\infty)$ and $v_1(t)$ be bounded outside a neighborhood of the
origin. The inequality
\begin{equation} \label{vav01com}
\ess_{r>0} v_2(r) H^{\ast}_{w} g(r) \leq C \ess_{r>0} v_1(r) g(r)
\end{equation}
holds for some $C>0$ for all non-negative and non-decreasing $g$ on $(0,\i)$ if and
only if
\begin{equation} \label{vav02com}
B:= \sup_{r>0} v_2(r)\int_r^{\infty} \left(1+\ln \frac{t}{r}\right) \, \frac{w(t) dt}{\ess_{t<s<\infty} v_1(s)}<\infty.
\end{equation}
Moreover, the value  $C=B$ is the best constant for  \eqref{vav01com}.
\end{thm}

\begin{rem}\label{rem2.3.}
In \eqref{vav01com} and \eqref{vav02com} it is assumed that $\frac{1}{\i}=0$ and $0 \cdot \i=0$.
\end{rem}

The following lemma is valid.
\begin{lem}\label{lem5.1.}
Let $0<\a<n$ and $b\in BMO(\Rn)$. Let $\Phi$ be a Young function and $\Psi$ defined, via its inverse, by setting, for all $t\in(0,\i)$, $\Psi^{-1}(t):=\Phi^{-1}(t)t^{-\a/n}$ and $1<a_{\Phi}\le b_{\Phi}<\i$ and $1<a_{\Psi}\le b_{\Psi}<\i$,
then the inequality
\begin{equation*}\label{eq5.1.}
\|[b,I_{\a}]f\|_{L_\Psi(B(x_0,r))} \lesssim \|b\|_{*} \, \frac{1}{\Psi^{-1}\big(r^{-n}\big)}
 \int_{2r}^{\i} \Big(1+\ln \frac{t}{r}\Big)\|f\|_{L_{\Phi}(B(x_0,t))}  \Psi^{-1}\big(t^{-n}\big) \frac{dt}{t}
\end{equation*}
holds for any ball $B(x_0,r)$ and for all $f\in L_{\Phi}^{\rm loc}(\Rn)$.

%Moreover, for $T_{b}$ bounded on $WL_{\Phi}(\Rn)$ the inequality
%\begin{equation}\label{eq3.5.WXC}
%\|T_{b}f\|_{WL_\Phi(B(x_0,r))} \lesssim \|a\|_{*} \, \frac{1}{\Phi^{-1}\big(r^{-n}\big)}
% \int_{2r}^{\i} \Big(1+\ln \frac{t}{r}\Big)\|f\|_{L_{\Phi}(B(x_0,t))}  \Phi^{-1}\big(t^{-n}\big) \frac{dt}{t}
%\end{equation}
%holds for any ball $B(x_0,r)$ and for all $f\in L_{\Phi}^{\rm loc}(\Rn)$.
\end{lem}

\begin{proof}
For arbitrary $x_0 \in\Rn$, set $B=B(x_0,r)$ for the ball
centered at $x_0$ and of radius $r$. Write $f=f_1+f_2$ with
$f_1=f\chi_{_{2B}}$ and $f_2=f\chi_{_{\dual (2B)}}$. Hence
$$
\left\|[b,I_{\a}]f\right\|_{L_\Psi(B)} \leq
\left\|[b,I_{\a}]f_1 \right\|_{L_\Psi(B)}+
\left\|[b,I_{\a}]f_2 \right\|_{L_\Psi(B)}.
$$
From the boundedness of $[b,I_{\a}]$ from $L_{\Phi}(\Rn)$ to $L_{\Psi}(\Rn)$ (see Theorem \ref{comorlicz})  %(\cite{BramCer}, Theorem 2.5)
it follows that
\begin{align*}
\|[b,I_{\a}]f_1\|_{L_\Psi(B)} & \leq
\|[b,I_{\a}]f_1\|_{L_\Psi(\Rn)}
\\
& \lesssim  \|b\|_{*} \, \|f_1\|_{L_\Phi(\Rn)} = \|b\|_{*} \, \|f\|_{L_\Phi(2B)}.
\end{align*}

For $x \in B$ we have
\begin{align*}
|[b,I_{\a}]f_2(x)| & \lesssim \int_{\Rn} \frac{|b(y)-b(x)|}{|x-y|^{n-\a}}|f(y)|dy
\\
&\thickapprox \int_{\dual (2B)} \frac{|b(y)-b(x)|}{|x_0-y|^{n-\a}}|f(y)| dy.
%\\
%& \le \int_{\dual (2B)} \frac{|a(y)-a_{B}|}{|x_0-y|^{n}} |f(y)|dy
%\\
%& + \int_{\dual (2B)} \frac{|a(x)-a_{B}|}{|x_0-y|^{n}} |f(y)| dy.
\end{align*}

Then
\begin{align*}
\|[b,I_{\a}]f_2\|_{L_\Psi(B)} & \lesssim
\left\|\int_{\dual (2B)} \frac{|b(y)-b(\cdot)|}{|x_0-y|^{n-\a}}|f(y)|dy\right\|_{L_\Psi(B)}
\\
&\lesssim \left\|\int_{\dual (2B)} \frac{|b(y)-b_{B}|}{|x_0-y|^{n-\a}}|f(y)|dy\right\|_{L_\Psi(B)}
\\
&\quad +\left\|\int_{\dual (2B)}
\frac{|b(\cdot)-b_{B}|}{|x_0-y|^{n-\a}}|f(y)|dy\right\|_{L_\Psi(B)}
\\
&=J_1+J_2.
\end{align*}
Let us estimate $J_1$.
\begin{align*}
J_1&= \frac{1}{\Psi^{-1}\big(r^{-n}\big)}\int_{\dual
(2B)}\frac{|b(y)-b_{B}|}{|x_0-y|^{n-\a}}|f(y)|dy
\\
&\thickapprox \frac{1}{\Psi^{-1}\big(r^{-n}\big)}\int_{\dual
(2B)}|b(y)-b_{B}||f(y)|\int_{|x_0-y|}^{\infty}\frac{dt}{t^{n+1-\a}}dy
\\
&\thickapprox \frac{1}{\Psi^{-1}\big(r^{-n}\big)} \int_{2r}^{\infty}\int_{2r\leq |x_0-y|\leq t}
|b(y)-b_{B}||f(y)|dy\frac{dt}{t^{n+1-\a}}
\\
&\lesssim \frac{1}{\Psi^{-1}\big(r^{-n}\big)} \int_{2r}^{\infty}\int_{B(x_0,t)}
|b(y)-b_{B}||f(y)|dy\frac{dt}{t^{n+1-\a}}.
\end{align*}

Applying H\"older's inequality, by Lemma \ref{Bmo-orlicz} and \eqref{propBMO}  we get
\allowdisplaybreaks
\begin{align*}
J_1 & \lesssim \frac{1}{\Psi^{-1}\big(r^{-n}\big)} \int_{2r}^{\infty}\int_{B(x_0,t)}
|b(y)-b_{B(x_0,t)}||f(y)|dy\frac{dt}{t^{n+1-\a}}
\\
&\quad + \frac{1}{\Psi^{-1}\big(r^{-n}\big)} \int_{2r}^{\infty}|b_{B(x_0,r)}-b_{B(x_0,t)}|
\int_{B(x_0,t)} |f(y)|dy\frac{dt}{t^{n+1-\a}}
\\
&\lesssim \frac{1}{\Psi^{-1}\big(r^{-n}\big)} \int_{2r}^{\infty}
\left\|b(\cdot)-b_{B(x_0,t)}\right\|_{L_{\widetilde{\Phi}}(B(x_0,t))} \|f\|_{L_\Phi(B(x_0,t))}\frac{dt}{t^{n+1-\a}}
\\
& \quad + \frac{1}{\Psi^{-1}\big(r^{-n}\big)} \int_{2r}^{\infty}|b_{B(x_0,r)}-b_{B(x_0,t)}|
\|f\|_{L_\Phi(B(x_0,t))}\Phi^{-1}\big(t^{-n}\big)\frac{dt}{t^{1-\a}}
\\
& \lesssim \|b\|_{*}\,\frac{1}{\Psi^{-1}\big(r^{-n}\big)}
\int_{2r}^{\infty}\Big(1+\ln \frac{t}{r}\Big)
\|f\|_{L_\Phi(B(x_0,t))}\Psi^{-1}\big(t^{-n}\big)\frac{dt}{t}.
\end{align*}
In order to estimate $J_2$ note that
\begin{align*}
J_2\thickapprox
\left\|b(\cdot)-b_{B}\right\|_{L_\Psi(B)}\int_{\dual (2B)}
\frac{|f(y)|}{|x_0-y|^{n-\a}}dy.
\end{align*}
By Lemma \ref{Bmo-orlicz}, we get
\begin{align*}
J_2\lesssim \|b\|_{*}\,\frac{1}{\Psi^{-1}\big(r^{-n}\big)}\int_{\dual (2B)}
\frac{|f(y)|}{|x_0-y|^{n-\a}}dy.
\end{align*}
Thus, by \eqref{sal00}
$$
J_2\lesssim \|b\|_{*}\,\frac{1}{\Psi^{-1}\big(r^{-n}\big)}
\int_{2r}^{\infty}\|f\|_{L_\Phi(B(x_0,t))}\Psi^{-1}\big(t^{-n}\big)\frac{dt}{t}.
$$
Summing $J_1$ and $J_2$ we get
\begin{equation*}
\|[b,I_{\a}]f_2\|_{L_\Psi(B)}
\lesssim \|b\|_{*}\,\frac{1}{\Psi^{-1}\big(r^{-n}\big)}
\int_{2r}^{\infty}\Big(1+\ln \frac{t}{r}\Big) \|f\|_{L_\Phi(B(x_0,t))}\Psi^{-1}\big(t^{-n}\big)\frac{dt}{t}.
\end{equation*}
Finally,
$$
\|[b,I_{\a}]f\|_{L_\Psi(B)}\lesssim \|b\|_{*}\,\|f\|_{L_\Phi(2B)}+
\|b\|_{*}\,\frac{1}{\Psi^{-1}\big(r^{-n}\big)}
\int_{2r}^{\infty}\Big(1+\ln \frac{t}{r}\Big) \|f\|_{L_\Phi(B(x_0,t))}\Psi^{-1}\big(t^{-n}\big)\frac{dt}{t},
$$
and the statement of Lemma \ref{lem5.1.} follows by \eqref{ves2F}.

%From the weak boundedness of $T_{b}$ and \eqref{ves2fd} it follows
%that
%\begin{align} \label{gfr9Com}
%\|T_{b}f_1\|_{WL_\Phi(B)} & \leq \|T_{b}f_1\|_{WL_\Phi(\Rn)} \notag
%\\
%& \lesssim \|a\|_{*}\,\|f_1\|_{L_\Phi(\Rn)} = \|a\|_{*}\,\|f\|_{L_\Phi(2B)}\notag
%\\
%& \lesssim \|a\|_{*}\,\frac{1}{\Phi^{-1}\big(r^{-n}\big)}\int_{2r}^{\i} \|f\|_{L_{\Phi}(B(x_0,t))} \Phi^{-1}\big(t^{-n}\big) \frac{dt}{t}.
%\end{align}
%%where the constant $C >0$ is independent of $f$.
%
%Then from \eqref{deckfV} and \eqref{gfr9Com} we get the inequality \eqref{eq3.5.WXC}.
\end{proof}

\begin{thm}\label{3.4.XcomT}
Let $0<\a<n$ and $b\in BMO(\Rn)$. Let $\Phi$ be a Young function and $\Psi$ defined, via its inverse, by setting, for all $t\in(0,\i)$, $\Psi^{-1}(t):=\Phi^{-1}(t)t^{-\a/n}$ and $1<a_{\Phi}\le b_{\Phi}<\i$ and $1<a_{\Psi}\le b_{\Psi}<\i$.  $(\varphi_1,\varphi_2)$ and $(\Phi, \Psi)$ satisfy the condition
\begin{equation}\label{eq3.6.VZfrMax}
\int_{r}^{\i}\Big(1+\ln \frac{t}{r}\Big) \es_{t<s<\i}\frac{\varphi_1(x,s)}{\Phi^{-1}\big(s^{-n}\big)}\Psi^{-1}\big(t^{-n}\big)\frac{dt}{t}  \le C \, \varphi_2(x,r),
\end{equation}
where $C$ does not depend on $x$ and $r$.

Then  the operator $[b,I_{\a}]$ is bounded from $M_{\Phi,\varphi_1}(\Rn)$ to
$M_{\Psi,\varphi_2}(\Rn)$. Moreover
$$\|[b,I_{\a}] f\|_{M_{\Psi,\varphi_2}} \lesssim  \|b\|_{*}\|f\|_{M_{\Phi,\varphi_1}}.$$
%Let $T_{b}$ be a sublinear operator satisfying condition \eqref{subcom} and bounded from $L_{\Phi}(\Rn)$ to $WL_{\Phi}(\Rn)$.
%
%Then  the operator $T_{b}$ is bounded from $M_{\Phi,\varphi_1}(\Rn)$ to
%$WM_{\Phi,\varphi_2}(\Rn)$. Moreover
%$$\|T_b f\|_{WM_{\Phi,\varphi_2}} \le  \|a\|_{*}\|f\|_{M_{\Phi,\varphi_1}}.$$
\end{thm}

\begin{proof}
The statement of Theorem \ref{3.4.XcomT} follows by Lemma \ref{lem5.1.} and Theorem \ref{thm3.2.com} in the same manner as in the proof of Theorem \ref{thm4.4.}.
\end{proof}

If we take $\Phi(t)=t^{p},\,\Psi(t)=t^{q},\,1< p,q<\i$ at Theorem \ref{3.4.XcomT} we get following corollary which was proved in \cite{GULAKShIEOT2012} (see, also \cite{GulShu}).
\begin{cor}\label{bounComRszGenMor}
Let $0<\a<n$, $1< p < \frac{n}{\a}$, $\frac{1}{q}=\frac{1}{p}-\frac{\a}{n}$, $b\in BMO(\Rn)$ and $(\varphi_{1},\varphi_{2})$ satisfy the condition
\begin{equation*}
\int_{r}^{\i}\Big(1+\ln \frac{t}{r}\Big)\frac{\es_{t<s<\i}\varphi_{1}(x,s)s^{\frac{n}{p}}}{t^{\frac{n}{q}+1}}dt\le C\varphi_{2}(x,r),
\end{equation*}
where $C$ does not depend on $x$ and $r$. Then $[b,I_{\a}]$ is bounded from $M_{p,\varphi_{1}}$ to $M_{q,\varphi_{2}}$.
\end{cor}

\

In the case $\varphi_1(x,r)=\frac{\Phi^{-1}\big(r^{-n}\big)}{\Phi^{-1}\big(r^{-\lambda_1}\big)}$, $\varphi_2(x,r)=\frac{\Psi^{-1}\big(r^{-n}\big)}{\Psi^{-1}\big(r^{-\lambda_2}\big)}$  from Theorem \ref{3.4.XcomT} we get the following Spanne type theorem for the boundedness of the operator $[b,I_{\a}]$ on Orlicz-Morrey spaces.
\begin{cor} \label{ggffddF}
Let $0<\a<n$, $0 \le \lambda_1, \lambda_2 < n$ and $b\in BMO(\Rn)$. Let also $\Phi$ be a Young function and $\Psi$ defined, via its inverse, by setting, for all $t\in(0,\i)$, $\Psi^{-1}(t):=\Phi^{-1}(t)t^{-\a/n}$, $1<a_{\Phi}\le b_{\Phi}<\i$, $1<a_{\Psi}\le b_{\Psi}<\i$ and $(\Phi, \Psi)$ satisfy the condition
\begin{equation*}
\int_{r}^{\i} \Big(1+\ln \frac{t}{r}\Big)  \frac{\Psi^{-1}\big(t^{-n}\big)}{\Phi^{-1}\big(t^{-\lambda_1}\big)}\frac{dt}{t}  \le C \, \frac{\Psi^{-1}\big(r^{-n}\big)}{\Psi^{-1}\big(r^{-\lambda_2}\big)},
\end{equation*}
where $C$ does not depend on $r$.
Then $[b,I_{\a}]$ is bounded from $M_{\Phi,\lambda_1}(\Rn)$ to $M_{\Psi,\lambda_2}(\Rn)$.
\end{cor}

\begin{rem}
If we take $\Phi(t)=t^{p},\,\Psi(t)=t^{q},\,1\le p,q<\i$ at Corollary \ref{ggffddF} we get Spanne type boundedness of $[b,I_{\a}]$, i.e. if $0<\a<n$, $1<p<\frac{n}{\a}$, $0<\lambda<n-\a p$, $\frac{1}{p}-\frac{1}{q}=\frac{\a}{n}$ and $\frac{\lambda}{p}=\frac{\mu}{q}$, then for $p>1$ the operator $[b,I_{\a}]$ is bounded from $M_{p,\lambda}(\Rn)$ to $M_{q,\mu}(\Rn)$ and for $p=1$, $[b,I_{\a}]$ is bounded from $M_{1,\lambda}(\Rn)$ to $WM_{q,\mu}(\Rn)$.
\end{rem}

\

{\bf Acknowledgements.} The authors would like to express their gratitude to the referees for his very valuable comments and suggestions.
{The research of V. Guliyev and F. Deringoz were partially supported by the grant of Ahi Evran University Scientific Research Projects (PYO.FEN.4003.13.003) and (PYO.FEN.4003-2.13.007).}

The author declare that there is no consist of interests regarding the publication of this article.

\

\end{document}